\documentclass[reqno]{amsart}
\usepackage{amsmath, amssymb, amsthm, epsfig, enumitem}
\usepackage{hyperref, latexsym}
\usepackage{url}
\usepackage[mathscr]{euscript}

\usepackage{color}
\usepackage{fullpage} 
\usepackage{setspace}

\def\today{\ifcase\month\or
  January\or February\or March\or April\or May\or June\or=
  July\or August\or September\or October\or November\or December\fi
  \space\number\day, \number\year}
\DeclareMathOperator{\sgn}{\mathrm{sgn}}

 \newtheorem{theorem}{Theorem}
  
 \newtheorem{lemma}[theorem]{Lemma}
 \newtheorem{proposition}[theorem]{Proposition}
 \newtheorem{corollary}[theorem]{Corollary}
 \theoremstyle{definition}
 
 \newtheorem*{example}{Example}
 \theoremstyle{remark}

\newtheorem{question}{Question}

 \newcommand{\R}{\mathbb{R}}

\newcommand{\var}{{\rm Var}}
\newcommand{\intav}[1]{\mathchoice {\mathop{\vrule width 6pt height 3 pt depth  -2.5pt
\kern -8pt \intop}\nolimits_{\kern -6pt#1}} {\mathop{\vrule width
5pt height 3  pt depth -2.6pt \kern -6pt \intop}\nolimits_{#1}}
{\mathop{\vrule width 5pt height 3 pt depth -2.6pt \kern -6pt
\intop}\nolimits_{#1}} {\mathop{\vrule width 5pt height 3 pt depth
-2.6pt \kern -6pt \intop}\nolimits_{#1}}}

\begin{document}

\title[BV Continuity for the uncentered maximal operator]{BV Continuity for the uncentered \\ Hardy--Littlewood maximal operator}
\author[Gonz\'{a}lez-Riquelme and Kosz]{Cristian Gonz\'{a}lez-Riquelme and Dariusz Kosz}
\subjclass[2010]{26A45, 42B25, 39A12, 46E35, 46E39, 05C12.}
\keywords{Maximal operators; continuity; bounded variation;}
\address{IMPA - Instituto de Matem\'{a}tica Pura e Aplicada\\
Rio de Janeiro - RJ, Brazil, 22460-320.}
\email{cristian@impa.br}
\address{ Faculty of Pure and Applied Mathematics,
	Wroc\l aw University of Science and Technology, 
	Wyb. Wyspia\'nskiego 27,
	50-370 Wroc\l aw, Poland.}
\email{dariusz.kosz@pwr.edu.pl}

\allowdisplaybreaks
\numberwithin{equation}{section}

\maketitle

\begin{abstract}
We prove the continuity of the map
$f \mapsto \widetilde{M}f$
from $BV(\mathbb{R})$ to itself, where $\widetilde{M}$ is the uncentered Hardy--Littlewood maximal operator. This answers a question of Carneiro, Madrid and Pierce.
\end{abstract}
\section{Introduction}
\subsection{A brief historical perspective and background}
The study of regularity properties for maximal operators has been an important topic of research over the last years. The most classical object in this context is the centered Hardy--Littlewood maximal operator, that for every $f\in L^{1}_{\text{loc}}(\mathbb{R}^d)$ and $x\in \mathbb{R}^d$ is defined as
$$Mf(x):=\sup_{r>0}\ \intav{B(x,r)}|f|,$$
where $\intav{B}g:=\frac{\int_{B}g}{m(B)}$ and $m$ is the $d$--dimensional Lebesgue measure. The uncentered Hardy--Littlewood maximal operator, denoted by $\widetilde{M}$, is defined analogously by taking the supremum over balls that contain the point
$x$ but that are not necessarily centered at $x$.

The program of studying the regularity of maximal functions was initiated by Kinnunen \cite{Kinnunen1997}, who proved that the map
\begin{align}\label{map}
f\mapsto Mf
\end{align}
is bounded on $W^{1,p}(\mathbb{R}^d)$ if $p>1.$ 
Since then, several authors have made many contributions in this direction, including connections with potential theory and partial differential equations  (see, e.g., \cite{BRS2018,BCHP2012,CFS2015,CS2013,KL1998,PPSS2017,RSW}).

An important direction has been the study of the continuity of the map \eqref{map} in several settings. Since it is not necessarily a sublinear map at the derivative level, its continuity does not follow from the boundedness. The first result answering a question of this kind was due to Luiro \cite{Luiro2007}, who proved that the map \eqref{map} is in fact continuous from $W^{1,p}(\mathbb{R}^d)$ to itself for each $p>1.$ 
\subsection{The endpoint case $p=1$}
This case is much more subtle. In fact, Kinnunen's result certainly does not hold because for every non trivial $f\in L^{1}(\mathbb{R}^d)$ we have $\widetilde{M}f\notin L^{1}(\mathbb{R}^d).$ Nevertheless, since we are interested at the derivative level of the maximal functions, the natural version of the problem at $p=1$ is to study the map 
\begin{align}\label{map2}
f\mapsto \nabla Mf,
\end{align}
for $f\in W^{1,1}(\mathbb{R}^d).$
In fact, the boundedness of the map \eqref{map2} 
from $W^{1,1}(\mathbb{R}^d)$ to $L^{1}(\mathbb{R}^d)$ is an important open problem for $d\ge 2.$ In dimension $d=1$ the answer is positive, as has been established for the uncentered version in \cite{Tanaka2002} and for the centered version in \cite{Kurka2010}. In \cite{Luiro2017}, Luiro proved that the map \eqref{map2} is bounded when restricted to radial data in the uncentered setting. Other boundedness results in the radial setting have been achieved in \cite{ BM2019,BM2019.2, CGR,GR,LM2017}.

An important work for our purposes is \cite{AP2007}, where the boundedness $$\var \big( \widetilde{M}f\big) \le \var(f),$$
for every $f\in BV(\mathbb{R}),$ was proved. Moreover, the authors proved that $\widetilde{M}f$ is absolutely continuous for $f\in BV(\mathbb{R}),$ providing then an improvement over the original function. 
\subsection{Continuity at the endpoint case}
In this article we are particularly interested in the continuity of the map \eqref{map2} at the endpoint $p=1.$ We notice that the methods outlined by Luiro \cite{Luiro2007} cannot be adapted to our case, since they depend on the fact that the map $$f\mapsto Mf$$ is bounded on $L^{p}(\mathbb{R}^d)$ for $p>1$.  
In \cite{CMP2017} the first continuity results at the endpoint were obtained. The authors proved, among other results, that 
the map 
\begin{align*}
f\mapsto \big (\widetilde{M}f\big)'
\end{align*}
is continuous from $W^{1,1}(\mathbb{R})$ to $L^1(\mathbb{R})$ (cf. \cite[Theorem 1]{CMP2017}). This result was recently extended to the $W^{1,1}_{\rm rad}(\R^d)$ setting (the subspace of $W^{1,1}(\R^d)$ consisting of radial functions) in \cite{CGRM}, where a more general approach, that can be applied in the context of other maximal functions, is proposed.

In \cite{CMP2017} the authors also consider the space $BV(\mathbb{R})$ endowed with the norm $\|f\|_{BV}:=|f(-\infty)|+\var (f).$  About this, they asked the following question:
\begin{question}{(Question B in \cite{CMP2017})}
 Is the map $\widetilde{M} : BV(\mathbb{R})\to BV(\mathbb{R})$ continuous?
\end{question}
This question, in case of being answered affirmatively, would provide a generalization of  \cite[Theorem 1]{CMP2017} (since $W^{1,1}(\mathbb{R})$ embeds isometrically in $BV(\mathbb{R})$). It is important to notice that, in general, the continuity in the $BV(\mathbb{R})$ setting is more delicate that in the $W^{1,1}(\mathbb{R})$ setting. An example of this is that in the fractional setting the analogue of \cite[Theorem 1]{CMP2017} holds (see \cite{Madrid2017}) but the answer to the analogue of the previous question is negative (see \cite[Theorem 3]{CMP2017}).

The main goal of the present manuscript is to answer this question. We prove the following.
\begin{theorem}\label{maintheorem}
The map $\widetilde{M}:BV(\mathbb{R})\to BV(\mathbb{R})$ is continuous.

\end{theorem}

Several of the arguments in \cite{CMP2017} (also the arguments in \cite{CGRM}) rely on the regularity of the original function, therefore they are not enough to conclude Theorem \ref{maintheorem}. In fact, the authors (also the authors of \cite{CGRM}) used in their work a reduction of the problem to the analogous question stated for ``lateral'' maximal operators, but in our case that reduction causes several problems. For instance, the one-sided maximal function of a function in $BV(\mathbb{R})$ is not necessarily continuous, which brings additional difficulties. Therefore, a different idea is required in order to achieve the result. Our methods, although inspired by \cite{CMP2017}, provide a new approach to this kind of problems.

Having taken into account the main result obtained here, we summarize the situation of the {\it endpoint continuity program} (originally proposed in \cite[Table 1]{CMP2017}) in the table below. The word YES in a box means that the continuity of the corresponding map has been proved, whereas the word NO means that it has been shown that it fails. The remaining boxes are marked as OPEN problems. We notice that after this work the only open problems in this program are the ones related to the centered Hardy--Littlewood maximal operator in a continuous setting.
\begin{table}[h]
\renewcommand{\arraystretch}{1.3}
\centering
\caption{Endpoint continuity program}
\label{Table-ECP}
\begin{tabular}{|c|c|c|c|c|}
\hline
 \raisebox{-1.3\height}{------------}&  \parbox[t]{2.8cm}{  $W^{1,1}-$continuity; \\ continuous setting} &  \parbox[t]{2.8cm}{  $BV-$continuity; \\ continuous setting} & \parbox[t]{2.6cm}{  $W^{1,1}-$continuity; \\ discrete setting} &  \parbox[t]{2.5cm}{  $BV-$continuity; \\ discrete setting} \\ [0.5cm]
\hline
 \parbox[t]{3.3cm}{ Centered classical \\ maximal operator} &  \raisebox{-0.8\height}{OPEN} & \raisebox{-0.8\height}{OPEN} & \raisebox{-0.8\height}{YES$^2$}  &\raisebox{-0.8\height}{YES$^4$}\\[0.5cm]
 \hline
  \parbox[t]{3.3cm}{ Uncentered classical \\ maximal operator} &  \raisebox{-0.8\height}{YES$^1$} &  \raisebox{-0.8\height}{YES: Theorem \ref{maintheorem}} &  \raisebox{-0.8\height}{YES$^2$} &  \raisebox{-0.8\height}{YES$^1$} \\[0.5cm]
 \hline
 \parbox[t]{3.3cm}{ Centered fractional \\ maximal operator} &  \raisebox{-0.8\height}{YES$^5$} &  \raisebox{-0.8\height}{NO$^1$} &  \raisebox{-0.8\height}{YES$^3$} & \raisebox{-0.8\height}{NO$^1$}  \\[0.5cm]
 \hline
\parbox[t]{3.3cm}{ Uncentered fractional \\ maximal operator} &  \raisebox{-0.8\height}{YES$^4$} &  \raisebox{-0.8\height}{NO$^1$} &  \raisebox{-0.8\height}{YES$^3$}  & \raisebox{-0.8\height}{NO$^1$} \\[0.5cm]
 \hline
\end{tabular}
\vspace{0.05cm}
\flushleft{
\ \ \footnotesize{$^1$ Result previously obtained in \cite{CMP2017}.\\
\ \ $^2$ Result previously obtained in \cite[Theorem 1]{CH2012}. \\
\ \ $^3$ Result previously obtained in \cite[Theorem 3]{CM2015}.\\
\ \ $^4$ Result previously obtained in \cite{Madrid2017}.\\
\ \ $^5$ Result previously obtained in \cite{BM2019.2}.
}}
\end{table}
\section{Proof of Theorem \ref{maintheorem}}
\subsection{Preliminaries}
In this subsection we develop the main tools required in our work.
We start by stating the following result which describes the behavior of the maximal function at infinity.
\begin{lemma}\label{infty}
Given $f\in BV(\mathbb{R})$ let $|f|(\infty):=\underset{x \to \infty}{\lim}|f|(x)$ and $|f|(-\infty):=\underset{x\to -\infty}{\lim}|f|(x).$ Then 
$$\lim_{x\to \infty}\widetilde{M}f(x)=\lim_{x\to -\infty}\widetilde{M}f(x)=c,$$
where $c=\max \{|f|(\infty),|f|(-\infty)\}$.
\end{lemma}
\begin{proof}
	Without loss of generality we assume that $f \geq 0$ and $c = f(\infty)$. Observe that
	$$
	\widetilde{M} f (x) \geq \lim_{r \rightarrow \infty} \intav{(x-1,x+r)}f 
	= c 
	$$ 
	holds for any $x \in \mathbb{R}$. Fix $\varepsilon > 0$ and let $N_0 > 0$ be such that $f(x) \leq c + \frac{\varepsilon}{2}$ for $|x| > N_0$. We choose $N_1 > N_0$ satisfying
	$$
	\frac{2N_0 \|f\|_\infty}{N_1 - N_0} \leq \frac{\varepsilon}{2}.
	$$  
	Consider $x_0$ satisfying $|x_0| > N_1$ and any interval $I \ni x_0$. If $|I| < N_1 - N_0$, then clearly
	$$
	\intav{I}f 
	\leq c + \frac{\varepsilon}{2}.
	$$
	On the other hand, if $|I| \geq N_1 - N_0$, then
	$$
	\intav{I}f
	\leq \frac{1}{|I|} \int_{I \cap [-N_0,N_0]^{\rm c}} f(x) \, dx + \frac{1}{N_1 - N_0} \int_{[-N_0,N_0]} f(x) \, dx \leq c + \varepsilon.
	$$
	Since $\varepsilon > 0$ is arbitrary, the claim follows.
\end{proof}

The next goal is to use the $BV(\mathbb{R})$ norm to control the difference between two $BV(\mathbb{R})$ functions or between their maximal functions at a given point $x$. The following estimates, although very basic, will be extremely useful later on.

\begin{lemma}\label{uniformconvergence}
	Let $f, g \in BV(\mathbb{R})$. Then 
$$
|f(x) - g(x)| \leq 2 \|f-g\|_{BV} \quad {\rm and } \quad \big|\widetilde{M} f(x) - \widetilde{M}g(x)\big| \leq 2 \|f-g\|_{BV} 
$$
hold for any $x \in \mathbb{R}$.
\end{lemma}

\begin{proof}
The first inequality follows since
$$
|f(x) - g(x)| \leq |(f(x) - g(x)) - (f(-\infty) - g(-\infty))| + |f(-\infty) - g(-\infty)|.
$$
Now, assume $\widetilde{M}f(x) \geq \widetilde{M} g(x)$. By the first part of the lemma for any $I \ni x$ we have
$$
\intav{I}|g| \geq \intav{I}|f| - \intav{I}|g-f| \geq \intav{I}|f| - 2 \|f-g\|_{BV}.
$$
Thus, $\widetilde{M} g(x) \geq \widetilde{M} f(x) - 2 \|f-g\|_{BV}$ and the second part follows as well. 
\end{proof}

Contrasting with the $W^{1,1}(\mathbb{R})$ setting (see \cite[Lemma 14]{CMP2017}), in our context to make the reduction to the case $f \geq 0$ is much more problematic. In order to deal with this issue we require several results describing the relations between $f$ and $|f|$.
 
In the following, for given $g\in BV(\mathbb{R})$ we define $\underset{y\uparrow x}{\lim}\ g(y)=:g(x^{-})$ and $\underset{y \downarrow x}{\lim}\ g(y)=:g(x^{+})$. Also, for each $-\infty \le a<b\le \infty$ we introduce the quantity 
$$
\var_{(a,b)}(g):= \sup \Big\{ \sum_{i=1}^{K} |g(a_{i}) - g(a_{i-1})| \Big\},
$$ 
where the supremum is taken over all $K \in \mathbb{N}$ and all sequences $a < a_0 < \dots < a_K < b$ (notice that if $g$ is continuous at $a$ and $b$, then the sequences satisfying $a = a_0 < \dots < a_K = b$ can be considered instead and the supremum will not change). For a given partition $\mathcal{P} = \{ a_0 < a_1 < \dots < a_K\}$ we denote $\var(g, \mathcal{P}) := \displaystyle \sum_{i=1}^{K} |g(a_{i}) - g(a_{i-1})|$. Finally, we write $D_{l}(g):=\{x\in \mathbb{R};g(x)\neq g(x^{-})\}$ and $D_{r}(g):=\{x\in \mathbb{R};g(x)\neq g(x^{+})\}.$

\begin{lemma}\label{modulusformula}
Fix $f\in BV(\mathbb{R}).$ Then for any $-\infty \le a<b\le \infty$ we have 
\begin{align*}
\var_{(a,b)}(f)-\var_{(a,b)}(|f|) & =\sum_{x\in D_{l}(f)\cap (a,b)}|f(x)-f(x^{-})|-\big||f|(x)-|f|(x^{-})\big| \\ & \quad +\sum_{x\in D_{r}(f)\cap (a,b)}|f(x)-f(x^{+})|-\big||f|(x)-|f|(x^{+})\big|.
\end{align*}
\end{lemma}
\begin{proof}
Fix $-\infty\le a<b\le \infty$. We write $D_{l}(f)\cap (a,b)=:\{x_{l,n};n\in \mathbb{N}\}$ and $D_{r}(f)\cap (a,b)=:\{x_{r,n};n\in \mathbb{N}\}$, assuming that both of these sets are infinite (the other cases can be treated very similarly). Given $\varepsilon>0$ we choose a partition $\mathcal{P} \subset (a,b)$ such that
$$\var(f,\mathcal{P})>\var_{(a,b)}(f)-\varepsilon$$
and $$\var(|f|,\mathcal{P})>\var_{(a,b)}(|f|)-\varepsilon.$$
Then for fixed $N\in \mathbb{N}$ we construct $\widetilde {\mathcal{P}}=\widetilde {\mathcal{P}}(N) \subset (a,b)$ by adding to $\mathcal{P}$ (if needed) some extra points. The procedure consists of the following three steps.
\begin{enumerate}[label=(\roman*)]
	\itemsep0.25em
    \item We set $\mathcal{P}_1=\mathcal{P}\cup \{x_{l,n},x_{r,n};n\le N\}.$
    \item For each $n\le N$ we choose $\widetilde{x}_{l,n}<x_{l,n}$ such that 
    $$\mathcal{P}_1\cap (\widetilde{x}_{l,n},x_{l,n})=\emptyset$$ and
    $|f(x_{l,n}^-)-f(\widetilde{x}_{l,n})|<2^{-n}\varepsilon.$
    Similarly, we choose $\widetilde{x}_{r,n}>x_{r,n}$ such that 
     $$\mathcal{P}_1\cap (x_{r,n},\widetilde{x}_{r,n})=\emptyset$$ and
    $|f(x_{r,n}^+)-f(\widetilde{x}_{r,n})|<2^{-n}\varepsilon.$
    Then we set $\mathcal{P}_2=\mathcal{P}_1\cup \{\widetilde{x}_{l,n},\widetilde{x}_{r,n}; n\le N\}.$
    \item For $K=K(\mathcal{P}_2)$, we let $\{\{x_k,y_k\};k\le K\}$ be the set of all pairs $\{x,y\}\subset \mathcal{P}_2$ satisfying $x<y$ with $(x,y)\cap \mathcal{P}_2=\emptyset$ and $f(x)f(y)<0,$ which are not of the form $\{\widetilde{x}_{l,n},x_{l,n}\}$ or $\{x_{r,n},\widetilde{x}_{r,n}\}$. Let $k\le K$. If there exists $z_{k}^{\circ}\in (x_k,y_k)$ such that $|f(z_k^{\circ})|<2^{-k}\varepsilon$, then we just add $z_k^{\circ}$ to $\mathcal{P}_2.$ If not, then at least one of the sets 
    $$I_{k,l} := (x_k,y_k]\cap \{z; \sgn(f(z)f(y_k))=\sgn(f(z^{-})f(x_k))=1\}$$ and
    $$I_{k,r} := [x_k,y_k)\cap \{z; \sgn(f(z^+)f(y_k))=\sgn(f(z)f(x_k))=1\}$$ must be non-empty (here $\sgn(x)$ is the usual sign function taking the value of $-1$, $0$, or $1$, if $x<0$, $x=0$, or $x>0$, respectively). Assume $I_{k,l} \neq \emptyset$ (the other case is similar) and choose $z_k \in I_{k,l}$. Then $z_k=x_{l,n}$ for some $n>N.$ We find $\widetilde{z}_{k}\in (x_k,z_k)$ such that $|f(\widetilde{z}_k)-f(z_k^{-})|<2^{-k}\varepsilon$ (in particular, we have $\sgn(f(x_k))=\sgn(f(\widetilde{z}_{k}))$), and add both $z_k$ and $\widetilde{z}_k$ to $\mathcal{P}_2.$ The above process terminates after $K$ steps and we denote the final collection of points by $\widetilde{\mathcal{P}}.$
\end{enumerate}
Having constructed $\widetilde{\mathcal{P}}$ we see that 
\begin{align*}
\var_{(a,b)}(f)-\var_{(a,b)}(|f|)&\ge \var(f,\widetilde{\mathcal{P}})-\var(|f|,\widetilde{\mathcal{P}})-\varepsilon\\
&\ge \sum_{n=1}^{N}|f(x_{l,n})-f(\widetilde{x}_{l,n})|-\big||f|(x_{l,n})-|f|(\widetilde{x}_{l,n})\big|\\
&\quad +\sum_{n=1}^{N}|f(x_{r,n})-f(\widetilde{x}_{r,n})|-\big||f|(x_{r,n})-|f|(\widetilde{x}_{r,n})\big|-\varepsilon\\
&\ge \sum_{n=1}^{N}|f(x_{l,n})-f(x_{l,n}^{-})|-\big||f|(x_{l,n})-|f|({x}_{l,n}^{-})\big|\\
&\quad +\sum_{n=1}^{N}|f(x_{r,n})-f(x_{r,n}^{+})|-\big||f|(x_{r,n})-|f|(x_{r,n}^{+})\big|-5\varepsilon.
\end{align*}
Also, we obtain
\begin{align*}
\var_{(a,b)}(f)-\var_{(a,b)}(|f|)&\le \var(f,\widetilde{\mathcal{P}})-\var(|f|,\widetilde{\mathcal{P}})+\varepsilon\\
&\le \sum_{n=1}^{\infty}|f(x_{l,n})-f(x_{l,n}^{-})|-\big||f|(x_{l,n})-|f|({x}_{l,n}^{-})\big|\\
&\quad +\sum_{n=1}^{\infty}|f(x_{r,n})-f(x_{r,n}^{+})|-\big||f|(x_{r,n})-|f|(x_{r,n}^{+})\big|+6\varepsilon,
\end{align*}
since the only terms that contributes to $\var(f,\widetilde{\mathcal{P}})-\var(|f|,\widetilde{\mathcal{P}})$ are those corresponding to the pairs $\{\widetilde{x}_{l,n},x_{l,n}\},$ $\{x_{r,n},\widetilde{x}_{r,n}\},$ $\{x_k,z_k^{\circ}\},$ $\{z_k^{\circ},y_k\}$ and $\{z_k,\widetilde{z_k}\}.$ Letting $N\to \infty$ and $\varepsilon \to 0,$ we obtain the claim. 
\end{proof}
Now, we use Lemma \ref{modulusformula} to show that the map $f \mapsto \var_{(a,b)}(|f|)$ is continuous from $BV(\mathbb{R})$ to $[0,\infty)$.
\begin{lemma}\label{variationconvergence}
Fix $f\in BV(\mathbb{R})$ and let $\{f_j;j\in \mathbb{N}\} \subset BV(\mathbb{R})$ be such that $\underset {j\to \infty}{\lim} \|f_j-f\|_{BV}=0.$ Then for any $-\infty\le a<b\le \infty$ we have 
$$\lim_{j\to \infty} \var_{(a,b)}(|f_j|)=\var_{(a,b)}(|f|).$$
\end{lemma}
\begin{proof}
It is possible to verify that $f_j\to f$ implies $\var_{(a,b)}(f_j)\to \var_{(a,b)}(f).$ Thus, it remains to show 
$$\lim_{j\to \infty} \var_{(a,b)}(f_j)-\var_{(a,b)}(|f_j|)=\var_{(a,b)}(f)-\var_{(a,b)}(|f|).$$ 
We define $\{x_{l,n};n\in \mathbb{N}\}$ and $\{x_{r,n};n\in \mathbb{N}\}$ as in the previous lemma. Given $\varepsilon>0$ we choose $N\in \mathbb{N}$ such that $$\sum_{n=N+1}^{\infty}|f(x_{l,n})-f(x_{l,n}^{-})|+|f(x_{r,n})-f(x_{r,n}^{+})|<\varepsilon.$$ 
We also denote $D_{l}^{j,N}:=D_{l}(f_j)\cap \{x_{l,1},\dots,x_{l,N}\}$ and $D_{r}^{j,N}:=D_{r}(f_j)\cap \{x_{r,1},\dots, x_{r,N}\}.$
By Lemma \ref{uniformconvergence} we have that for $j$ big enough
$$\Big|\Big(\sum_{D_{l}^{j,N}} |f_{j}(x)-f_{j}(x^{-})|-\big||f_j|(x)-|f_j|(x^{-})\big|-|f(x)-f(x^{-})|+\big||f|(x)-|f|(x^{-})\big| \Big)\Big|<\varepsilon$$ and
$$\Big|\Big(\sum_{D_{r}^{j,N}} |f_{j}(x)-f_{j}(x^{+})|-\big||f_j|(x)-|f_j|(x^{+})\big|-|f(x)-f(x^{+})|+\big||f|(x)-|f|(x^{+})\big|\Big)\Big|<\varepsilon.$$
Moreover, we have 
\begin{align*}
    0 & \leq \underset{_{x\in \big(D_{l}(f_j)\cap (a,b)\big )\setminus D_{l}^{j,N}}}{\sum}|f_{j}(x)-f_{j}(x^{-})|-\big||f_j|(x)-|f_{j}|(x^{-})\big| \\
    & \le \underset{_{x\in \big(D_{l}(f_j)\cap (a,b)\big )\setminus D_{l}^{j,N}}}{\sum}|f_{j}(x)-f_{j}(x^{-})|\\
    &\le \underset{_{x\in \big(D_{l}(f_j)\cap (a,b)\big )\setminus D_{l}^{j,N}}}{\sum}|f(x)-f(x^{-})|+4 \|f-f_j\|_{BV}
    <2\varepsilon
\end{align*}
and, similarly, 
\begin{align*}
 0 \leq \underset{_{x\in \big(D_{r}(f_j) \cap (a,b)\big )\setminus D_{r}^{j,N}}}{\sum}|f_{j}(x)-f_{j}(x^{+})|-\big||f_j|(x)-|f_{j}|(x^{+})\big|<2\varepsilon.
\end{align*}
Finally, we observe that $\{x_{l,1},\dots,x_{l,N}\}\subset D_{l}(f_j)$ and $\{x_{r,1},\dots, x_{r,N}\}\subset D_{r}(f_j)$ for $j$ big enough, by the uniform convergence. Letting $\varepsilon \to 0$ (and thus $N\to \infty$) and applying Lemma \ref{modulusformula}, we obtain the claim.
\end{proof}
 
Let us now take a closer look at the properties of the maximal operator. Recall that the total variation of $\widetilde{M}f$ can be controlled by the total variation of $f$. There is also a local version of this principle, where we focus on an interval $(a,b)$. However in this case some boundary terms must be included. Thus, to avoid the possibility that $f$ behaves badly at $a$ or $b$, we use its adjusted version $\overline{|f|}$ defined by $$\overline{|f|}(x):=\underset{I\ni x; m(I)\to 0}{\lim \sup}\intav{I}|f|.$$ 
It is known that $\overline{|f|}$ is upper semicontinuous and that $\overline{|f|}\le \widetilde{M}f$ (see \cite[Lemma 3.3]{AP2007}).

\begin{lemma}\label{aldazperezlazaro}
Fix $f\in BV(\mathbb{R}).$ Given $-\infty\le a<b\le \infty,$ we have
$$\var_{(a,b)}\big (\widetilde{M}f\big)\le \var_{(a,b)}\big(\overline{|f| }\big)+\big|\widetilde{M}f(a)-\overline{|f|}(a)\big|+\big|\widetilde{M}f(b)-\overline{|f|}(b)\big|.$$
\end{lemma}
\begin{proof}
This follows by a slight modification of the proof of \cite[Lemma 3.9]{AP2007}.
\end{proof}

The next result gives us the uniform control (with respect to $j$) on the behavior of $\big (\widetilde{M}f_j\big)'$ near infinity, provided that $\{ f_j\}_{j \in \mathbb{N}}$ is a converging sequence in $BV(\mathbb{R})$. This, in turn, allows one to restrict the attention to a bounded interval, while dealing with the total variations of the maximal functions $\widetilde{M}f_j$. We point out that it is also possible to proceed without this reduction, but then for all considered functions the extended domain $[-\infty, \infty]$ should be used instead of $\mathbb{R}$.   

\begin{lemma}\label{controlatinfty}
Fix $f\in BV(\mathbb{R})$ and let $\{f_j;j\in \mathbb{N}\}\subset BV(\mathbb{R})$ be such that $\underset{j\to \infty}{\lim}\|f_j-f\|_{BV}=0.$ Then for any $\varepsilon>0$ there exist $-\infty<a<b<\infty$ such that $$\int_{\mathbb{R} \setminus (a,b)}\left|\big (\widetilde{M}f\big)'\right|<\varepsilon$$ and
$$\int_{\mathbb{R} \setminus (a,b)}\left|\big (\widetilde{M}f_j\big)'\right|<\varepsilon,$$
for every $j$ big enough.  
\end{lemma}
\begin{proof}
We prove that there exists $b<\infty$ such that 
$\displaystyle\int_{(b,\infty)}\big|\big (\widetilde{M}f\big)'\big|<\varepsilon,$ the symmetric case is treated analogously. 
First we deal with the case where $\widetilde{M}f(\infty)>|f|(\infty)$. Assume that $\widetilde{M}f(\infty)-|f|(\infty)>4\varepsilon$. Let us take $b$ big enough such that (we use here Lemma \ref{uniformconvergence}) we have $\big|\widetilde{M}f(x)-\widetilde{M}f(\infty)\big|<\varepsilon$ and $\big||f|(x)-|f|(\infty)\big|< \varepsilon$ for every $x\in (b,\infty).$  Therefore, for $j$ big enough such that $\||f_j|-|f|\|_{\infty}\le \frac{\varepsilon}{2}$ and $\|\widetilde{M}f_j-\widetilde{M}f\|_{\infty}\le \frac{\varepsilon}{2},$ for each $y \in (b,x)$ we have $\widetilde{M}f_{j}(y)\ge \widetilde{M}f_{j}(x).$ This is the case because any interval $I\ni x$ satisfying $\intav{I}|f_j|>\widetilde{M}f_j(x)-\frac{\varepsilon}{2}$ contains $y,$ since if $I\subset (y,\infty)$, in particular $I\subset (b,\infty),$ and then $$\intav{I}|f_j|\le \intav
{I}|f|+\frac{\varepsilon}{2}\le |f|(\infty)+\frac{3\varepsilon}{2}\le \widetilde{M}f(\infty)-2\varepsilon\le \widetilde{M}f(x)-\varepsilon\le \widetilde{M}f_j(x)-\frac{\varepsilon}{2}.$$ Therefore $$\int_{(b,\infty)}\left|\big (\widetilde{M}f_j\big)'\right|=\widetilde{M}f_j(b)-\widetilde{M}f_j(\infty)\le \big|\widetilde{M}f(b)-\widetilde{M}f_j(b)\big|+\big|\widetilde{M}f_j(\infty)-\widetilde{M}f(\infty)\big|+\big|\widetilde{M}f(b)-\widetilde{M}f(\infty)\big|\le 3\varepsilon$$ for $j$ big enough, from where we conclude this case.  

Now, we deal with the case where $|f|(\infty)=\widetilde{M}f(\infty).$ 
By Lemma \ref{aldazperezlazaro} and \cite[Lemma 3.3]{AP2007}, assuming that $b$ is a continuity point for $f_j$, we obtain \begin{align*}
\int_{(b,\infty)}\Big|\big (\widetilde{M}f_j\big)'\Big|&\le \var_{(b,\infty)}\big(\overline{|f_j|}\big)+\big|\widetilde{M}f_j(b)-\overline{|f_j|}(b)\big|+\big|\widetilde{M}f_j(\infty)-\overline{|f_j|}(\infty)\big|\\
&\le \var_{(b,\infty)}(|f_j|)+\big|\widetilde{M}f_j(b)-|f_j|(b)\big|+\big|\widetilde{M}f_j(\infty)-|f_j|(\infty)\big|.
\end{align*}
The analogous is obtained for $f$ instead of $f_j$. Let us assume that $b$ is a continuity point for $f$ and every $f_j$, such that $\var_{(b,\infty)}\big(|f|\big)<\varepsilon.$ By Lemma \ref{variationconvergence} we have \begin{align}\label{2varepsilon}
\var_{(b,\infty)}(|f_j|)<2\varepsilon, 
\end{align}
for $j$ big enough. Also, $$\big|\widetilde{M}f_j(b)-|f_j|(b)\big|\le \big|\widetilde{M}f(b)-|f|(b)\big|+|\widetilde{M}f_j(b)-\widetilde{M}f(b)|+\big||f_j|(b)-|f|(b)\big|.$$ If $b$ is big enough to have $\big|\widetilde{M}f(b)-|f|(b)\big|< \varepsilon$, then by Lemma \ref{uniformconvergence} we get $\big|\widetilde{M}f_j(b)-|f_j|(b)\big|< 2 \varepsilon$ and $\big|\widetilde{M}f_j(\infty)-|f_j|(\infty)\big|<\varepsilon$ for $j$ big enough. Combining this with \eqref{2varepsilon} concludes the proof. 
\end{proof}

Before we prove our key result regarding the variation of maximal functions, we need the following definition. A given partition $\mathcal{P}=\{a_0<a_1 < \dots <a_n\},$ with $n\ge 2,$ has property (V) with respect to $f$ if for each $i\in \{0,1,\dots, n-2\},$ we have $\sgn(f(a_{i+2})-f(a_{i+1}))\cdot \sgn(f(a_{i+1})-f(a_{i}))<0$. 
\begin{proposition}\label{keyproposition}
Fix $f\in BV(\mathbb{R})$ and let $\{f_j;j\in \mathbb{N}\}\subset BV(\mathbb{R})$ be such that $\underset{j\to \infty}{\lim}\|f_j-f\|_{BV}=0.$ Then 
$$\var_{(-\infty, \infty)} \big (\widetilde{M}f_j\big)\to \var_{(-\infty, \infty)}\big (\widetilde{M}f\big).$$
\end{proposition}
\begin{proof}
By Lemma \ref{controlatinfty} it is enough to prove that $\var_{(a,b)}\big (\widetilde{M}f_j\big)\to \var_{(a,b)}\big (\widetilde{M}f\big)$ for every interval $(a,b) \subset \mathbb{R}$ with both $a$ and $b$ being points of continuity for $f$ and every $f_j$. In the following we fix $-\infty<a<b<\infty$ satisfying such assumption.
Observe that Lemma \ref{uniformconvergence} and Fatou's lemma imply
$$\underset{j\to \infty}{\lim \inf}\ \var_{(a,b)}\big (\widetilde{M}f_j\big)\ge \var_{(a,b)}\big (\widetilde{M}f\big).$$
Now, we prove the remaining inequality, that is, 
$$\underset{j\to \infty}{\lim \sup}\ \var_{(a,b)}\big (\widetilde{M}f_j\big)\le \var_{(a,b)}\big (\widetilde{M}f\big).$$
Given $\varepsilon>0$ we show that $$\var_{(a,b)}\big (\widetilde{M}f_j\big)<\var_{(a,b)}\big (\widetilde{M}f\big)+4\varepsilon$$ holds if $j$ is big enough. 
Let $\mathcal{P}=\{a=a_0<a_1<\dots<a_K=b\}\subset \mathbb{R},$ $K\in \mathbb{N},$ be a partition satisfying 
\begin{align*}
 \var\big(|f|,\mathcal{P}\big)>\var_{(a,b)}\big(|f|\big)-\varepsilon  
\end{align*}
and 
\begin{align*}
 \var\big(\widetilde{M}f,\mathcal{P}\big)>\var_{(a,b)}\big (\widetilde{M}f\big)-\varepsilon. \end{align*}
Also, by the uniform convergence and Lemma \ref{variationconvergence} we conclude that 
\begin{align}\label{aproximation1}
    \var\big(|f_j|,\mathcal{P}\big)>\var_{(a,b)}\big(|f_j|\big)-2 \varepsilon
\end{align}
and \begin{align}\label{aproximation2}
\var\big(\widetilde{M}f_j,\mathcal{P}\big)>\var_{(a,b)}\big (\widetilde{M}f\big)-2 \varepsilon
\end{align}
hold for $j$ big enough. Now, we take $\widetilde{\mathcal{P}}=\widetilde{\mathcal{P}}(j)$ such that $\mathcal{P}\subset \widetilde{\mathcal{P}}\subset [a,b]$ and 
$$\var\big(\widetilde{M}f_{j},\widetilde{\mathcal{P}}\big)>\var_{(a,b)}\big (\widetilde{M}f_j\big)-\varepsilon.$$
Without loss of generality we can assume that $\widetilde{\mathcal{P}}$ is such that for each $i\in \{1,\dots,K\}$ the set $\widetilde{\mathcal{P}}\cap [a_{i-1},a_i]=\{a_{i-1}=a_{i,0}<\dots<a_{i,n_i}=a_i\}$ satisfies property (V) with respect to $\widetilde{M}f_{j}$ unless it consists of two elements. 
For each such $i$ we denote $\widetilde{\mathcal{P}}_{i}=\{a_{i,1},\dots,a_{i,n_{i}-1}\}$ and claim that it is possible to find a partition $\widetilde{\mathcal{P}}_{i}^{*}=\{a_{i,1}^{*},\dots,a_{i,n_i-1}^{*}\}\subset (a_{i-1},a_i)$ such that
\begin{align*}
\var\Big(|f_j|, \widetilde{\mathcal{P}}_{i}^{*}\Big)-\var\Big(|f_j|,\{a_{i,1}^*,a_{i,n_i-1}^{*}\}\Big)>\var\left(\widetilde{M}f_{j}, \widetilde{\mathcal{P}}_{i}\right)-\var\left(\widetilde{M}f_j,\{a_{i,1},a_{i,n_{i}-1}\}\right)-\frac{\varepsilon}{K}.    
\end{align*}
Indeed, for $n_i \leq 2$ we use the convention that all the variation terms above are equal to $0,$ so the inequality holds (we set $\widetilde{\mathcal{P}}_{i}^{*} = \emptyset$ or $\widetilde{\mathcal{P}}_{i}^{*} = \{ a_{i,1}\}$ if $n=1$ or $n=2$, respectively). It remains to consider the case $n_i\ge 3$ in which property $(V)$ is guaranteed. We assume that $\widetilde{M}f_{j}(a_{i,0})<\widetilde{M}f_{j}(a_{i,1})$ (the opposite case can be treated analogously). Then $\widetilde{\mathcal{P}}_{i}^{*}$ shall be chosen in such a way that given $k\in \{1,\dots, n_{i}-1\}$ we have 
$$|f_{j}|(a_{i,k}^{*}) >\max \Big\{\widetilde{M}f_{j}(a_{i,k-1}),\widetilde{M}f_j(a_{i,k})-\frac{\varepsilon}{2n_iK}, \widetilde{M}f_j(a_{i,k+1})\Big \},$$ 
for $k$ odd, and 
$$|f_j|(a_{i,k}^{*})\le \widetilde{M}f_{j}(a_{i,k})$$
for $k$ even.
We describe in detail the procedure for selecting the points $a_{i,k}^{*}.$ If $k$ is odd, then we find an interval $I\ni a_{i,k}$ such that 
$$
\intav{I}|f_j|>\max \Big\{\widetilde{M}f_{j}(a_{i,k-1}),\widetilde{M}f_j(a_{i,k})-\frac{\varepsilon}{2n_iK}, \widetilde{M}f_j(a_{i,k+1})\Big \}.
$$
Clearly, $I\subset (a_{i,k-1},a_{i,k+1})$ and we can find $a_{i,k}^{*}\in I$ satisfying $|f_{j}|(a_{i,k}^{*}) \ge \intav{I}|f_j|.$ For $k$ even we take $I=(a_{i,k-1}^{*},a_{i,k+1}^{*})$ if $k \neq n_i-1$ or $I=(a_{i,n_{i}-2}^{*},a_i)$ otherwise. Since
$\intav{I}|f_j|\le \widetilde{M}f_{j}(a_{i,k}),$
there exists $a_{i,k}^{*}\in I$ satisfying $|f_{j}|(a_{i,k}^{*})\le \widetilde{M}f_j(a_{i,k}).$ We note that the appropriate configuration of the sets $I$ guarantees that the inequalities $a_{i-1}<a_{i,1}^*<\dots<a_{i,n_{i}-1}^*<a_{i}$ hold. 

Observe, also, that the partition $\{a_{i,1}^* < \dots < a_{i,n_i-1}^*\}$ either consists of $2$ elements or satisfies property (V) with respect to $|f_j|$. Thus, regardless of which case occurs, we obtain
$$\var\Big(|f_j|, \widetilde{\mathcal{P}}_{i}^{*}\Big)-\var\Big(|f_j|,\{a_{i,1}^{*},a_{i,n_{i}-1}^{*}\}\Big)=\sum_{k=1}^{n_{i}-1}\alpha_{k}|f_{j}|(a_{i,k}^{*}),$$
where $\alpha_{k}=2(-1)^{k+1}$ for $k\in \{2,\dots, n_{i}-2\}$ and $\alpha_k\in \left \{0,2(-1)^{k+1}\right \}$ for $k\in \{1,n_{i}-1\}$ (the boundary values depend on $\sgn\big(|f_{j}|(a_{i,1}^{*})-|f_j|(a_{i,n_i}^{*})\big)$ and the parity of $n_i$). Similarly,
$$\var\left(\widetilde{M}f_j,\widetilde{\mathcal{P}}_{i}\right)-\var\left(\widetilde{M}f_j,\{a_{i,1}, a_{i,n_{i}-1}\}\right)\le \sum_{k=1}^{n_{i}-1} \alpha_{k}\widetilde{M}f_{j}(a_{i,k})$$
(we eventually change the sign of the second term on the left-hand side in order to get the boundary coefficients equal to $\alpha_1$ and $\alpha_{n_i-1}$). Consequently, the claim follows since for each $k$ we have 
$$\alpha_{k}\left(|f_{j}|(a_{i,k}^{*})-\widetilde{M}f_j(a_{i,k})\right)\ge \frac{-\varepsilon}{n_i K}.$$
Now, we apply the claim in order to get the following estimate
\begin{align*}
\var_{(a,b)}(|f_j|)-\var(|f_j|,\mathcal{P})&\ge 
\var\Big(|f_j|, \mathcal{P} \cup \bigcup_{i=1}^K \widetilde{\mathcal{P}}_{i}^{*}\Big) - \var\Big(|f_j|, \mathcal{P} \cup \bigcup_{i=1}^K \{a_{i,1}^{*},a_{i,n_i-1}^{*}\} \Big)
\\
&=\sum_{i=1}^{K}\var\Big(|f_j|,\widetilde{\mathcal{P}}_{i}^{*}\Big)-\var\Big(|f_j|,\{a_{i,1}^{*},a_{i,n_{i}-1}^{*}\}\Big)\\
&\ge \sum_{i=1}^{K} \var\left(\widetilde{M}f_j,\widetilde{\mathcal{P}}_{i}\right)-\var\left(\widetilde{M}f_j,\{a_{i,1},a_{i,n_i-1}\}\right)-\frac{\varepsilon}{K}\\
&\ge \var\Big(\widetilde{M}f_{j},\widetilde{\mathcal{P}}\Big)-\var\Big(\widetilde{M}f_{j},\widetilde{\widetilde{\mathcal{P}}}\Big)-\varepsilon,
\end{align*}
where $\widetilde{\widetilde{\mathcal{P}}}:=\{a_{i,k};i\in \{1,\dots, K\},k\in \{0,1,n_i-1,n_i\}\}.$
In particular, we note that $\widetilde{\widetilde{\mathcal{P}}}$ consists of at most $3K+1$ elements and thus 
$$\var\Big(\widetilde{M}f_j,\widetilde{\widetilde{\mathcal{P}}}\Big)<\var\Big(\widetilde{M}f,\widetilde{\widetilde{\mathcal{P}}}\Big)+12 K \|f_j-f\|_{BV}<\var_{(a,b)}\big (\widetilde{M}f\big)+\varepsilon$$
follows by Lemma \ref{uniformconvergence} for $j$ big enough. Combining the above inequalities with \eqref{aproximation2}, we arrive at
$$\var_{(a,b)}(|f_j|)-\var(|f_j|,\mathcal{P})>\var_{(a,b)}\big (\widetilde{M}f_j\big)-\var_{(a,b)}\big (\widetilde{M}f\big)-4\varepsilon,$$
which, in view of \eqref{aproximation1}, gives
$$\var_{(a,b)}\big (\widetilde{M}f_j\big)<\var_{(a,b)}\big (\widetilde{M}f\big)+6\varepsilon,$$
provided that $j$ is big enough. Consequently, $\underset{j\to \infty}{\lim} \var_{(a,b)}\big (\widetilde{M}f_j\big)=\var_{(a,b)}\big (\widetilde{M}f\big).$
\end{proof}
Having obtained Proposition \ref{keyproposition} we continue with the remaining tools required. Our general purpose in the next few lemmas is to get more information about the derivative of maximal function. In particular, we are interested in studying the behavior of the sequence $\big\{ \big (\widetilde{M}f_j\big)'(x)\big \}_{j \in \mathbb{N}}$ for a given point $x$.   

\begin{lemma}\label{accumulation}
Fix $f\in BV(\mathbb{R})$ and let $\{f_j;j\in \mathbb{N}\}\subset BV(\mathbb{R})$ be such that $\underset{j\to \infty}{\lim}\|f_j-f\|_{BV}=0.$ If $\intav{I_{x,j}}|f_j|=\widetilde{M}f_j(x)$ with $I_{x,j}\ni x$, and $\chi_{I_{x,j}}\to \chi_{I}$ a.e. with $0<m(I)<\infty$, then we have $\intav{I}|f|=\widetilde{M}f(x).$
\end{lemma}
\begin{proof}
Follows a slight modification in \cite[Lemma 12]{CMP2017}.
\end{proof}
 Let us now define $E:=\big\{x\in \mathbb{R}; \widetilde{M}f(x)>\overline{|f|}(x)\big\}.$ Since $\widetilde{M}f$ is absolutely continuous and $\overline{|f|}$ is upper semicontinuous, 
we have that $E$ is open. 
We notice that if $x\in E\setminus \big\{x; \widetilde{M}f(x)=\widetilde{M}f(\infty)\big\}$, then there exists a finite interval $I\ni x$ such that $\intav{I}|f|(x)=\widetilde{M}f(x).$ Indeed, there exists a sequence $\{(a_k,b_k)\}_{k\in \mathbb{N}}$ such that $x\in (a_k,b_k)$ and $\intav{(a_k,b_k)}|f|\to \widetilde{M}f(x).$ Since $\widetilde{M}f(x)>\big\{\overline{|f|}(x),\widetilde{M}f(\infty)\big\}$, we have $\{ b_k-a_k; k \in \mathbb{N} \} \subset (\epsilon, \epsilon^{-1})$ for some $\epsilon > 0$. Thus, by taking a subsequence (if required), we get $a_k\to a$ and $b_k\to b$, with $b-a \in (0, \infty).$ By the boundedness of $f$ we conclude that $\intav{[a,b]}|f|=\widetilde{M}f(x).$ Also, let us observe that $\max \big \{\intav{[a,x]}|f|,\intav{[x,b]}|f|\big \}\ge \intav{[a,b]}|f|,$ therefore $\widetilde{M}f(x)=\intav{[a,x]}|f|$ or $\widetilde{M}f(x)=\intav{[x,b]}|f|.$  

The next result states that for a.e. $x \in E$ the derivative of the maximal function $\widetilde{M}f$ can be described by an explicit formula.  
\begin{lemma}\label{formula}
Let $f\in BV(\mathbb{R})$. Assume that $\widetilde{M}f$ is differentiable and $|f|$ is continuous at $x$ (that happens a.e. because $\widetilde{M}f$ and $|f|$ have bounded variation). Let us suppose that $x\in E$ is such that there exists an interval $I_x\ni x$ with $m(I_x)<\infty$ such that $\intav{I_x}|f|=\widetilde{M}f(x)$ and $I_x \subset [x, \infty)$ or $I_x \subset (- \infty, x]$. Then 
$$
\big (\widetilde{M}f\big)'(x)= \begin{cases}
\frac{\int_{I_x}|f|}{(m(I_x))^2}-\frac{|f|(x)}{m(I_x)} = \frac{1}{m(I_x)} \left( \widetilde{M}f(x) - |f|(x) \right) \quad {\rm if }  \ I_x \subset [x, \infty), \\
\frac{|f|(x)}{m(I_x)} - \frac{\int_{I_x}|f|}{(m(I_x))^2} = \frac{1}{m(I_x)} \left(|f|(x) - \widetilde{M}f(x) \right) \quad { \rm otherwise.}
\end{cases}
$$
Also, if $\widetilde{M}f(x)=\widetilde{M}f(\infty)$, then we have $\big (\widetilde{M}f\big)'(x)=0$.
\end{lemma}
\begin{proof}
The last claim follows because $x$ is a local minimum of $\widetilde{M}f.$ Assume without loss of generality that $I_x=(x,a_x)$, $a_x > x$ (the other case is similar). We have, for $h>0$, that 
$$
\frac{\widetilde{M}f(x)-\widetilde{M}f(x-h)}{h} \le \frac{\intav{(x,a_x)}|f|-\intav{(x-h,a_x)}|f|}{h}
=\frac{\frac{\int_{x}^{a_x}|f|}{a_x-x}-\frac{\int_{x}^{a_x}|f|+\int_{x-h}^{x}|f|}{a_x-x+h}}{h} 
\to \frac{\int_{I_x}|f|}{(a_x-x)^2}-\frac{|f|(x)}{a_x-x}
$$
as $h\to 0.$ Therefore $\big (\widetilde{M}f\big)'(x)\le \frac{\int_{I_x}|f|}{(a_x-x)^2}-\frac{|f(x)|}{a_x-x}.$ 
Also, for $h>0$ we have 
$$
\frac{\widetilde{M}f(x+h)-\widetilde{M}f(x)}{h} \ge \frac{\frac{\int_{x+h}^{a_x}|f|}{a_x-x-h}-\frac{\int_{x}^{a_x}|f|}{a_x-x}}{h} 
\to \frac{\int_{I_x}|f|}{(a_x-x)^2}-\frac{|f|(x)}{a_x-x}
$$
as $h\to 0.$ This concludes the proof.
\end{proof}
Now, we can use the obtained formula to prove the following result regarding pointwise convergence. 
\begin{lemma}\label{pointwiseconvergence}
Fix $f\in BV(\mathbb{R})$ and let $\{f_j;j\in \mathbb{N}\}\subset BV(\mathbb{R})$ be such that $\underset{j\to \infty}{\lim}\|f_j-f\|_{BV}=0.$ Then 
$$\big (\widetilde{M}f_j\big)'\to \big (\widetilde{M}f\big)'$$
a.e. in $E.$ 
\end{lemma}
\begin{proof}
The claim is trivial if $E$ has measure zero, so assume this is not the case. We define $E_j$ as the analogue of $E$ for $f_j.$
Let us take $x\in E$ such that $\widetilde{M}f_j$ and $\widetilde{M}f$ are differentiable at $x$ for every $j$ and $f$ and $f_j$ is continuous at $x$. By the uniform convergence we have that $x\in E_j$ for $j$ big enough. We also make the following observation. If there are intervals $I_{x,j}\ni x$ such that $\intav{I_{x,j}}|f_j|=\widetilde{M}f_j(x)$, then the quantities $m(I_{x,j})$ are bounded below uniformly. Indeed, if for a sequence $\{j_k\}_{k\in \mathbb{N}}$ we have $m(I_{x,j_k})\to 0$, then we would have $\intav{I_{x,j_k}}|f_{j_k}|\to |f|(x) < \widetilde Mf(x)$ by the uniform convergence and continuity of $f$ at $x$, contradicting the pointwise convergence of the maximal functions. 

Assume first that $x\in E\setminus \big \{y;\widetilde{M}f(y)=\widetilde{M}f(\infty)\big \}$ 
and take $\varepsilon>0$ such that $\widetilde{M}f(x)>\widetilde{M}f(\infty)+2\varepsilon.$ Then for $j$ big enough we have $\widetilde{M}f_j(x)>\widetilde{M}f_j(\infty)+\varepsilon.$ Also, there exists $N>|x|$ such that for $j$ big enough and each $y \in \mathbb{R}\setminus [-N,N]$ we have $|f_j|(y)<\widetilde{M}f(\infty)+\varepsilon < \widetilde{M}f_j(x)$. We can observe then that $I_{x,j}\subset [-N,N]$ for $j$ big enough. 
Let us assume that we have $\delta>0$ and a sequence $\{j_k\}_{k\in \mathbb{N}}$ such that  \begin{align}\label{epsiloncontradiction}
\left|\big (\widetilde{M}f_{j_k}\big)'(x) -\big (\widetilde{M}f\big)'(x)\right|>\delta.
\end{align}
Without loss of generality assume that $I_{j_k}=(x,a_{j_k})$ (the other case is treated analogously). 
Since $x<a_{j_k}<N$, there exists a subsequence (that we also denote by $j_k$) such that $a_{j_k}\to a\in [x,N].$ Moreover, in view of the previous observation, we have $a \neq x$. Thus, Lemma \ref{accumulation} gives $\intav{(x,a)}|f|=\widetilde{M}f(x)$ and consequently, in view of Lemma \ref{formula}, we obtain $\big (\widetilde{M}f\big)'(x)=\frac{\int_{(x,a)}|f|}{(a-x)^2}-\frac{|f|(x)}{a-x}$. Also, $\big (\widetilde{M}f_{j_k}\big)'(x)=\frac{\int_{I_{x,j}}|f_{j_k}|}{(a_{j_k}-x)^2}-\frac{|f_{j_k}|(x)}{a_{j_k}-x}$ holds. However, by the uniform convergence we have    $$\frac{\int_{I_{x,j}}|f_{j_k}|}{(a_{j_k}-x)^2}-\frac{|f_{j_k}|(x)}{a_{j_k}-x}\to \frac{\int_{(x,a)}|f|}{(a-x)^2}-\frac{|f|(x)}{a-x},$$ reaching a contradiction with \eqref{epsiloncontradiction}. Thus, we conclude this case.

Now, if $\widetilde{M}f(x)=\widetilde{M}f(\infty),$ then by Lemma \ref{formula} we have $\big (\widetilde{M}f\big)'(x)=0.$ Also, if for a subsequence $j_k$ we have $\widetilde{M}f_{j_k}(x)=\widetilde{M}f_{j_k}(\infty)$, then $\big (\widetilde{M}f_{j_k}\big)'(x)=0.$ Therefore, this subcase follows and we can assume that $x\in E_j\setminus \big\{\widetilde{M}f_j(x)=\widetilde{M}f_j(\infty)\big\}.$ It is now enough to prove that 
$\frac{\int_{I_{x,j}}|f_j|}{(a_j-x)^2}-\frac{|f_j|(x)}{a_j-x}\to 0.$ Let us suppose that for some $\delta>0$ and a subsequence $j_k$ we have $\big|\big (\widetilde{M}f_{j_k}\big)'\big|>\delta.$ As before, we assume the case $I_{j_k}=(x,a_{j_k})$. We claim that there exists $C(\delta,f)>0$ such that for $j_k$ big enough we have $m(I_{x,j_k})<C(\delta,f)<\infty.$ Indeed, in view of $$\frac{2 \|f_{j_k}\|_{\infty}}{m(I_{x,j_{k}})}>\left|\frac{\int_{I_{x,j_{k}}}|f_{j_k}|}{(a_{j_k}-x)^2}-\frac{|f_{j_k}|(x)}{a_{j_k}-x}\right|>\delta,$$ 
we have $\frac{2\|f_{j_k}\|_{\infty}}{\delta}>m(I_{x,{j_k}})$ and thus $\|f_{j_k}\|_{\infty} \to \|f\|_{\infty}$ gives our claim. 
Now, since $m(I_{x,j_{k}})<C(\delta,f),$ we have that $a_{j_k}\in (x,x+C(\delta,f))$ for $j$ big enough. Consequently, there exists a subsequence (that we also denote by $j_k$) such that $a_{j_k}\to a$ for some $a\in (x,x+C(\delta,f)].$ Then by Lemma \ref{accumulation} we have that $\intav{(x,a)}|f|=\widetilde{M}f(x).$ Therefore, Lemma \ref{formula} gives
$$\frac{\int_{(x,a)}|f|}{(a-x)^2}-\frac{|f|(x)}{a-x}=\big (\widetilde{M}f\big)'(x)$$
and the left-hand side must be equal to $0$.
Since we have $$\big (\widetilde{M}f_{j_k}\big)'(x)= \frac{\int_{I_{x,j_{k}}}|f_{j_k}|}{(a_{j_k}-x)^2} -\frac{|f_{j_k}|(x)}{a_{j_k}-x} \to \frac{\int_{(x,a)}|f|}{(a-x)^2}-\frac{|f|(x)}{a-x}=0$$ by the uniform convergence, we reach a contradiction. This concludes the proof.  
\end{proof}

It remains to take a look at the set $C:=\mathbb{R}\setminus E$. 


\begin{lemma}\label{zeroderivative}
Let $f\in BV(\mathbb{R})$. Then for a.e. $x\in C$ we have $\big (\widetilde{M}f\big)'(x)=0.$
\end{lemma}
\begin{proof}
Assume that $\overline{|f|}$ and $\widetilde{M}f$ are differentiable at $x$ (this happens a.e. because $\overline{|f|}$ and $\widetilde{M}f$ have bounded variation). 
Then, since $\widetilde{M}f(x)=\overline{|f|}(x)$ and $\widetilde{M}f\ge \overline{|f|}$, we have $\big (\widetilde{M}f\big)'(x)=\overline{|f|}'(x).$ Now, assume, in order to get a contradiction, that $\overline{|f|}'(x)>0$ (the other case is analogous). Then there exist $h_0,L>0$ such that $\overline{|f|}(x+h)\ge \overline{|f|}(x)+Lh$ for every $0<h<h_0.$ Thus, for a.e. $0<h<h_0$ we have $|f|(x+h)\ge \overline{|f|}(x)+Lh$, which implies $\widetilde{M}f(x)\ge \frac{\int_{0}^{h_0}\overline{|f|}(x)+Lh}{h_0}=\overline{|f|}(x)+\frac{Lh_0}{2}>\overline{|f|}(x),$
a contradiction.
\end{proof}
Combining the previous results we obtain the following.
\begin{corollary}\label{corolario}
Fix $f\in BV(\mathbb{R})$ and let $\{f_j;j\in \mathbb{N}\} \subset BV(\mathbb{R})$ be such that $\underset{j\to \infty}{\lim}\|f_j-f\|_{BV}=0.$ Then 
$$\left \|\left(\big (\widetilde{M}f_j\big)'-\big (\widetilde{M}f\big)'\right)\chi_{E}\right \|_1\to 0.$$
\end{corollary}
\begin{proof}
By the classic Brezis--Lieb lemma \cite{Brezislieb}, the boundedness of the map $f\mapsto \widetilde{M}f$ from $BV(\mathbb{R})$ to itself and Lemma \ref{pointwiseconvergence}, we just need to prove the following,
\begin{align}\label{blcorollary}
\left\|\big (\widetilde{M}f_j\big)'\chi_E\right\|_1\to \left\|\big(\widetilde{M}f\big)'\chi_E\right\|_1.
\end{align}
By Fatou's Lemma, Proposition \ref{keyproposition} and Lemma \ref{zeroderivative}, we have
$$
\int_{E}\left|\big (\widetilde{M}f\big)'\right| \le \underset{j\to \infty}{\lim \inf} \int_{E}\left|\big (\widetilde{M}f_j\big)'\right| \le \underset{j\to \infty}{\lim \sup} \int_{E}\left|\big (\widetilde{M}f_j\big)'\right|\le \underset{j\to \infty}{\lim} \int_{\mathbb{R}}\left|\big (\widetilde{M}f_j\big)'\right|=\int_{\mathbb{R}}\left|\big (\widetilde{M}f\big)'\right|=\int_{E}\left|\big (\widetilde{M}f\big)'\right|, 
$$
from where \eqref{blcorollary} follows.
\end{proof}

\subsection{Proof of Theorem \ref{maintheorem}} Finally, we are ready to prove the main result. In what follows $C_j$ denotes the set analogous to $C$ defined for $f_j$ instead of $f.$
\begin{proof}Since by Lemmas \ref{infty} and \ref{uniformconvergence} we have $\widetilde{M}f_j(-\infty)\to \widetilde{M}f(-\infty)$, it remains to prove that 
$$\big (\widetilde{M}f_j\big)'\to \big (\widetilde{M}f\big)'$$
in $L^{1}(\mathbb{R}).$
We make the following claim 
\begin{align}\label{claim1}
\int_{C\cap E_j}\left|\big (\widetilde{M}f_j\big)'\right|\to 0.
\end{align}
Indeed, by Proposition \ref{keyproposition}, Lemma \ref{zeroderivative} and Corollary \ref{corolario} we have 
\begin{align*}
\underset{j\to \infty}{\lim} \int_{E}\left|\big (\widetilde{M}f_j\big)'\right|=\int_{E}\left|\big (\widetilde{M}f\big)'\right|=\int_{\mathbb{R}}\left|\big (\widetilde{M}f\big)'\right|&\ge \underset{j\to \infty}{\lim \sup }\left(\int_{E_j\cap C}\left|\big (\widetilde{M}f_j\big)'\right|+\int_{E}\left|\big (\widetilde{M}f_j\big)'\right|\right)\\
&=\underset{j\to \infty}{\lim \sup} \int_{E_j\cap C}\left|\big (\widetilde{M}f_j\big)'\right|+\underset{j\to \infty}{\lim} \int_{E}\left|\big (\widetilde{M}f_j\big)'\right|
\end{align*}
and the claim follows. Consequently, by \eqref{claim1} and Corollary \ref{corolario} we get
\begin{align*}
\int_{\mathbb{R}}\left|\big (\widetilde{M}f_j\big)'-\big (\widetilde{M}f\big)'\right|&=\int_{C\cap E_j}\left|\big (\widetilde{M}f_j\big)'-\big (\widetilde{M}f\big)'\right|+\int_{C\cap C_j}\left|\big (\widetilde{M}f_j\big)'-\big (\widetilde{M}f\big)'\right| 
+\int_{E}\left|\big (\widetilde{M}f_j\big)'-\big (\widetilde{M}f\big)'\right|\\
&=\int_{C\cap E_j}\left|\big (\widetilde{M}f_j\big)'\right|+\int_{E}\left|\big (\widetilde{M}f_j\big)'-\big (\widetilde{M}f\big)'\right|\to 0
\end{align*}
as $j\to \infty,$ from where we conclude our result.
\end{proof}

\subsection{Concluding remarks} We end our discussion by showing that the assumptions $f,f_j\in BV(\mathbb{R})$ are important, not only $f-f_j\in BV(\mathbb{R}).$ 
\begin{example}
	Let $A = \cup_{k=1}^{\infty} (4k-2, 4k)$ and take
	$$
	f = \chi_{(-\infty, 0) \cup A}, \quad {\rm and} \quad f_n = f + \frac{1}{n} \chi_{(0, 4n+2)}. 
	$$
	Then we have $\|f_n - f\|_{BV} \rightarrow 0$, while $\| \widetilde{M}f_n - \widetilde{M}f \|_{BV} \not\rightarrow 0$.
\end{example}

\noindent Indeed, the first claim is obvious and for the second one we argue as follows. We observe that $\widetilde{M} f \equiv 1$ and $\widetilde{M} f_n(x) = 1 + \frac{1}{n}$ for $x \in \{ 3, 7, \dots, 4n-1 \}$. Moreover, if $n \geq 3$, then for any $x \in \{ 1, 5, \dots, 4n + 1 \}$ we have 
$$
\widetilde{M} f_n(x) \leq \max \Big\{1, \frac{2}{3} + \frac{1}{n} \Big\} = 1,
$$
which is due to the fact that for any interval $I \ni x$ we have
$
|I \cap A \cap (0, 4n+2)| \leq \frac{2}{3} |I \cap (0, 4n+2)|.
$
Thus, for $\mathcal{P}_n = \{1, 3, \dots, 4n+1 \}$ we have $\var\big(\widetilde{M} f_n - \widetilde{M} f, \mathcal{P}_n\big) \geq 2n \cdot \frac{1}{n} \not\rightarrow 0$.
\section*{Acknowledgements.}
Cristian Gonz\'{a}lez-Riquelme was supported by CAPES-Brazil. Dariusz Kosz was supported by the National Science Centre of Poland, project no. 2016/21/N/ST1/01496. The
authors are thankful to Emanuel Carneiro for helpful discussions during the preparation of this manuscript. 
\bibliography{Reference}
\bibliographystyle{amsplain} 
\end{document}